\documentclass{conm-p-l}

\copyrightinfo{2010}{}

\usepackage{graphicx}

\newtheorem{theorem}{Theorem}[section]
\newtheorem{lemma}[theorem]{Lemma}

\theoremstyle{definition}

\theoremstyle{remark}
\newtheorem{remark}[theorem]{Remark}

\numberwithin{equation}{section}

\newcommand{\tu}{\tilde{u}}

\newcommand{\e}{\epsilon}

\newcommand{\vv}{\vec{v}}
\newcommand{\vu}{\vec{u}}

\newcommand{\dl}{\delta}

\renewcommand{\th}{\theta}

\newcommand{\al}{\alpha}

\newcommand{\pa}{\partial}

\newcommand{\om}{\omega}

\newcommand{\na}{\nabla}
\newcommand{\lag}{\langle}
\newcommand{\rag}{\rangle}
\newcommand{\tx}{\tilde{x}}

\newcommand{\non}{\nonumber}

\begin{document}

\title[Recurrence in 2D Inviscid Channel Flow]{Recurrence in 2D Inviscid Channel Flow}

\author{Y. Charles Li}
\address{Department of Mathematics, University of Missouri, 
Columbia, MO 65211, USA}
\curraddr{}
\email{liyan@missouri.edu}
\thanks{}

%\author{Author Two}
%\address{Mathematical Research Section, School of Mathematical Sciences,
%Australian National University, Canberra ACT 2601, Australia}

%    General info
\subjclass{Primary 37, 76; Secondary 35, 34}
\date{}

\dedicatory{}


\begin{abstract}
I will prove a recurrence theorem which says that any $H^s$ ($s>2$) solution to the 2D inviscid channel flow returns repeatedly to an arbitrarily small $H^0$ neighborhood. Periodic boundary condition is imposed along the stream-wise direction. The result is an extension of an early result of the author [Li, 09] on 2D Euler equation under periodic boundary conditions along both directions.
\end{abstract}

\maketitle

%\section*{}
%This is an example of an unnumbered first-level heading.

%\specialsection*{This is a Special Section Head}
%This is an example of a special section head%
%%%%%%%%%%%%%%%%%%%%%%%%%%%%%%%%%%%%%%%%%%%%%%%%%%%%%%%%%%%%%%%%%%%%%%%%
%\footnote{Here is an example of a footnote. Notice that this footnote
%text is running on so that it can stand as an example of how a footnote
%with separate paragraphs should be written.
%\par
%And here is the beginning of the second paragraph.}%
%%%%%%%%%%%%%%%%%%%%%%%%%%%%%%%%%%%%%%%%%%%%%%%%%%%%%%%%%%%%%%%%%%%%%%%%

%\section{This is a numbered first-level section head}
%This is an example of a numbered first-level heading.

%\subsection{This is a numbered second-level section head}
%This is an example of a numbered second-level heading.

%This is an example of an unnumbered second-level heading.

%\subsubsection{This is a numbered third-level section head}
%This is an example of a numbered third-level heading.

%\subsection*{This is an unnumbered second-level section head}
%This is an example of an unnumbered second-level heading.

%\subsubsection{This is a numbered third-level section head}
%This is an example of a numbered third-level heading.

%\subsubsection*{This is an unnumbered third-level section head}
%This is an example of an unnumbered third-level heading.

\section{Introduction}

The recurrent nature of the solutions to the 2D Euler equation depends 
crucially upon the geometry of the domain containing the fluid. Here 
the specific geometry (rather than topology) is fundamental. We are 
interested in three types of domains:
\begin{itemize}
\item Periodic Cell $D_p = [0,\ell_x] \times [0, \ell_y]$, which is a 
period cell in $\mathbb{R}^2$ where the fluid flow is under periodic boundary 
conditions in both $x$ and $y$ directions with periods $\ell_x$ and $\ell_y$.
\item Channel Cell $D_c = [0,\ell_x] \times [0, a]$, which is a 
channel cell of the channel $\mathbb{R} \times [0, a]$ where the fluid 
flow is under periodic boundary condition in $x$ direction with period 
$\ell_x$, and non-penetrating conition along the walls $y =0, a$. 
\item Annulus $D_a = \{ (r, \th ) \ | \ R_1 \leq r \leq R_2 \}$ in 
polar coordinates, where the fluid 
flow is under non-penetrating conition along the walls $r =R_1, R_2$.
\end{itemize}

For the fluid domain $D_p$,
the solutions to the 2D Euler equation are recurrent somewhere along 
the orbit in the $L^2$ norm of velocity \cite{Li09a}. Below we will show 
that this is also true for the fluid domain $D_c$. For the fluid domain $D_a$,
we believe that this is no longer true. In the $D_a$ case, the enstrophy 
is not equivalent to the $W^{1,2}$ norm of velocity. For example,
\begin{equation}
(u,v) = \frac{1}{x^2+y^2} (-y, x) 
\label{PR}
\end{equation}
is an irrotational steady solution of the 2D Euler equation. Its vorticity 
is zero, while its $W^{1,2}$ norm is non-zero. The argument of \cite{Li09a}
for recurrence depends crucially upon the equivalence of the enstrophy  
to the $W^{1,2}$ norm of velocity. 

On the other hand, for the fluid domain $D_a$, there is an initial condition 
such that for
any other initial condition in a $C^1$ vorticity neighborhood, the solution  
never returns to the $C^1$ vorticity neighborhood \cite{Li09b} \cite{Nad91}.
The inner and outer circles of the annulus play a fundamental role in the 
the construction of this example. It is the multiconnectness nature of the 
annulus that makes the example possible. The irrotational steady solution 
(\ref{PR}) is also fundamental in the construction of the example. For the 
$D_p$ and $D_c$ cases, we do not believe the non-recurrent example of the 
$D_a$ case exists, since the irrotational solution (\ref{PR}) does not exist.

The recurrence in $L^2$ norm of velocity is a result of two main ingredients:
(1). the conservation of energy and enstropy, (2). the equivalence of the 
enstrophy to the $W^{1,2}$ norm of velocity. The first ingredient is also 
credited for generating the well-studied phenomenon, the so-called forward 
energy cascade and inverse enstrophy cascade, see e.g. \cite{Rut98}.

\section{Recurrence in 2D Inviscid Channel Flow}

The 2D inviscid channel flow is governed by
\begin{eqnarray}
& & \frac{\pa u}{\pa t} + u \frac{\pa u}{\pa x} + v \frac{\pa u}{\pa y}
= - \frac{\pa p}{\pa x}, \label{E1} \\
& & \frac{\pa v}{\pa t} + u \frac{\pa v}{\pa x} + v \frac{\pa v}{\pa y}
= - \frac{\pa p}{\pa y}, \label{E2} \\
& & \frac{\pa u}{\pa x} + \frac{\pa v}{\pa y} = 0; \label{E3} 
\end{eqnarray}
subject to the boundary condition
\begin{equation}
v =0 \quad \text{at } y =a, b. 
\label{BC}
\end{equation}
Here we also pose periodic boundary condition along $x$-direction (of period
$L_x$). Thus
\begin{eqnarray}
& & u = \sum_{n=-\infty}^{+\infty} u_n(y) e^{in\al x} , \quad 
v = \sum_{n=-\infty}^{+\infty} v_n(y) e^{in\al x} , \label{vft} \\
& & p = \sum_{n=-\infty}^{+\infty} p_n(y) e^{in\al x} ; \non
\end{eqnarray}
where $\al = 2\pi / L_x$. Our first goal here is to reach a setting that 
the spatial averages of $u$ and $v$ are zero. This setting is not necessary for 
proving the recurrence, but it is an interesting fact itself. By the 
incompressibility condition (\ref{E3}),
\[
\frac{\pa v_0}{\pa y} = 0.
\]
Using the boundary condition (\ref{BC}), one gets
\begin{equation}
v_0(y) = 0. 
\label{v}
\end{equation}
Thus the spatial average of $v$ is zero. Denote by $\lag u \rag$ the 
spatial average of $u$
\[
\lag u \rag = \frac{1}{(b-a)L_x} \int_a^b \int_0^{L_x} u \ dx dy
= \frac{1}{b-a} \int_a^b u_0(y) \ dy.
\]
The inviscid channel flow (\ref{E1})-(\ref{E3}) together with its boundary 
conditions are invariant under the transformation: 
\[
u = \tu + \lag u \rag , \quad  \tx = x - \lag u \rag t .
\]
Together with (\ref{v}), without loss of generality, we can assume that 
the spatial averages of $u$ and $v$ are zero. A simple calculation shows that 
the kinetic energy and enstrophy
\[
E = \int_a^b \int_0^{L_x} (u^2+v^2) \ dx dy , \quad 
G = \int_a^b \int_0^{L_x} \om^2 \ dx dy
\]
where $\om = \pa_x v - \pa_y u$ is the vorticity, are invariant under the 
inviscid channel flow. Next we will show that $G$ is equivalent to the 
$W^{1,2}$ norm of velocity. 

\begin{lemma}
\[
\int_a^b \int_0^{L_x} \left [ (\pa_xu)^2+(\pa_yu)^2+ (\pa_xv)^2+(\pa_yv)^2
\right ]\ dx dy = G.
\]
\label{equ}
\end{lemma}
\begin{proof}
By the incompressibility condition (\ref{E3}),
\[
\int_a^b \int_0^{L_x} \left [ (\pa_xu)^2+(\pa_yv)^2
\right ]\ dx dy = -2 \int_a^b \int_0^{L_x} (\pa_xu) (\pa_yv) \ dx dy .
\]
Notice that
\begin{eqnarray*}
& & \int_a^b \int_0^{L_x} \left [ (\pa_xu) (\pa_yv)-(\pa_yu) (\pa_xv) \right ] \ dx dy \\
& & =\int_a^b \int_0^{L_x} \left [ \pa_y (v\pa_xu) - \pa_x (v\pa_yu)\right ] \ dx dy  = 0 ,
\end{eqnarray*}
where the first term vanishes due to the boundary condition (\ref{BC}) and the second term vanishes due to the periodicity along $x$-direction. Using the above two facts, we have
\begin{eqnarray*}
& & \int_a^b \int_0^{L_x} \left [ (\pa_xu)^2+(\pa_yu)^2+ (\pa_xv)^2+(\pa_yv)^2
\right ]\ dx dy \\
&=& \int_a^b \int_0^{L_x} \left [ (\pa_yu)^2+ (\pa_xv)^2 - 2 (\pa_xu) (\pa_yv)
\right ]\ dx dy \\
&=& \int_a^b \int_0^{L_x} \left [ (\pa_yu)^2+ (\pa_xv)^2 - 2 (\pa_yu) (\pa_xv)
\right ]\ dx dy = G .
\end{eqnarray*}
\end{proof}

\begin{lemma}
For any $C>0$, the set 
\[
S=\left \{ \vv = (u,v) \ \bigg | \  G = \int_a^b \int_0^{L_x} \om^2 \ dx dy \leq C \right \}
\]
is compactly embedded in $L^2(D_c)$ of $\vv$, $D_c = [0, L_x]\times [a,b]$. That is, the closure of $S$ in 
$L^2(D_c)$ is a compact subset of $L^2(D_c)$ .
\label{compe}
\end{lemma}

\begin{remark}
By the fact (\ref{v}),
\[
\| v \|_{L^2(D_c)} \leq C \| \na v \|_{L^2(D_c)}.
\]
By the Poincar\'e inequality for $u_0(y)$ and the fact that the $y$-average of 
$u_0(y)$ is zero, 
\[
\| u \|_{L^2(D_c)} \leq C \| \na u \|_{L^2(D_c)} .
\]
Without the facts of (\ref{v}) and the vanishing of the $y$-average of 
$u_0(y)$, we can still carry out the rest argument of the paper by replacing 
$G$ with $E+G$ in the above definition of $S$.
The above lemma is the well-known Rellich lemma. In a simpler language, an enstrophy ball is compactly embedded in the kinetic energy space. The proof below 
follows that of \cite{Fol76}. $\quad \Box$
\end{remark}
\begin{proof}
By Lemma \ref{equ},
\[
\int_a^b \int_0^{L_x}  |\na \vv |^2 \ dx dy \leq C.
\]
Starting from the representation (\ref{vft}), we expand $\vv_n (y)$ into a Fourier integral 
\[
\vv_n (y) = \int_{-\infty}^{+\infty} \vv_{n\xi} e^{i\xi y} \ d \xi ,
\]
where (ignoring a constant factor)
\[
\vv_{n\xi} = \int_a^b \vv_n (y)e^{-i\xi y} \ dy.
\]
Let $\{ \vv^{\ k} \}$ be a sequence in $S$. 
\[
\pa_\xi \vv^{\ k}_{n\xi} =  \int_a^b (-iy)\vv^{\ k}_{n}(y) e^{-i\xi y} \ dy.
\]
Thus
\[
|\vv^{\ k}_{n\xi}|, |\pa_\xi \vv^{\ k}_{n\xi}| \leq C_1 \| \vv^{\ k} \|_{L^2(D_c)},
\]
where $C_1$ is a constant only dependent on $a$ and $b$. 
For any $n$, by the Arzel\`a-Ascoli theorem, there is a subsequence 
$\{ \vv^{\ k_j}_n \}$ which converges uniformly on compact sets of $\xi$. 
For a different $n$, we can start from the subsequence $\{ \vv^{\ k_j}_n \}$
and find a further uniformly convergent subsequence. Thus there is a 
uniformly convergent subsequence $\{ \vv^{\ k_j}_{n\xi } \}$ on $|\xi |
\leq A$, $|n| \leq N$ for any finite $A$ and $N$. Next we show that 
$\{ \vv^{\ k} \}$ forms a Cauchy sequence in $L^2(D_c)$.
\begin{eqnarray*}
& & \| \vv^{\ k_j} - \vv^{\ k_m} \|^2_{L^2(D_c)}
= \sum_{n=-\infty}^{+\infty} \int_{-\infty}^{+\infty} \left |\vv^{\ k_j}_{n\xi }
- \vv^{\ k_m}_{n\xi }\right |^2 \ d\xi \\
&=& \sum_{|n| \leq N} \int_{-A}^{+A} \left |\vv^{\ k_j}_{n\xi }
- \vv^{\ k_m}_{n\xi }\right |^2 \ d\xi
+\sum_{|n| \leq N} \int_{|\xi | > A} \left |\vv^{\ k_j}_{n\xi }
- \vv^{\ k_m}_{n\xi }\right |^2 \ d\xi \\
&+& \sum_{|n| > N} \int_{-\infty}^{+\infty} \left |\vv^{\ k_j}_{n\xi }
- \vv^{\ k_m}_{n\xi }\right |^2 \ d\xi
\leq \sup_{|n| \leq N,|\xi | \leq A}\left |\vv^{\ k_j}_{n\xi }
- \vv^{\ k_m}_{n\xi }\right |^2 4NA \\
&+&A^{-2}\sum_{|n| \leq N} \int_{|\xi | > A} |\xi |^2\left |\vv^{\ k_j}_{n\xi }
- \vv^{\ k_m}_{n\xi }\right |^2 \ d\xi \\
&+& N^{-2}\sum_{|n| > N} n^2\int_{-\infty}^{+\infty} \left |\vv^{\ k_j}_{n\xi }
- \vv^{\ k_m}_{n\xi }\right |^2 \ d\xi \\
&\leq& 4NA \sup_{|n| \leq N,|\xi | \leq A}\left |\vv^{\ k_j}_{n\xi }
- \vv^{\ k_m}_{n\xi }\right |^2 + (A^{-2}+N^{-2}) \| \vv^{\ k_j} -  
\vv^{\ k_m} \|_{H^1(D)} .
\end{eqnarray*}
For any $\e >0$, choose $A$ and $N$ large enough so that the second term is 
less than $\e /2$ for all $j$ and $m$. Then choose $j$ and $m$ large enough 
so that the first term is less than $\e /2$. Thus, $\{ \vv^{\ k} \}$ forms a Cauchy sequence in $L^2(D_c)$. Finally let $\{ \vv^{\ k} \}$ be a sequence of the accumulation 
points of $S$ in $L^2(D_c)$. Then we can find a sequence $\{ \vu^{\ k} \}$ in $S$ such that 
\[
\| \vv^{\ k} - \vu^{\ k}   \|_{L^2(D_c)} < 1/k.
\]
Let $\{ \vu^{\ k_j} \}$ be the convergent subsequence, then $\{ \vv^{\ k_j} \}$ is also a convergent subsequence. Thus the closure of $S$ in $L^2(D_c)$ is a compact subset of $L^2(D_c)$.
\end{proof}

The inviscid channel flow (\ref{E1})-(\ref{E3}), together with the boundary 
condition (\ref{BC}) and periodicity along $x$-direction, is globally 
well-posed in $H^s(D_c)$ ($s>2$) \cite{TW97} \cite{Tem75} \cite{Kat67}. We 
have the following recurrence theorem.
\begin{theorem}
For any $\vv_0 \in H^s(D_c)$ ($s>2$), any $\dl >0$, and any $T>0$; there is a $\vv_* 
\in H^s(D_c)$ such that
\[
F^{m_jT}(\vv_0 ) \in B^0_\dl (\vv_* )=\{ \vv \in H^s(D_c) \  | \  \| \vv -\vv_*  \|_{H^0(D_c)}
< \dl \}
\]
where $\{ m_j \}$ is an infinite sequence of positive integers, and $F^t$ is the evolution operator 
of the inviscid channel flow.
\end{theorem}
\begin{proof}
Choose the $C$ in Lemma \ref{compe} to be
\[
2\int_a^b \int_0^{L_x}|\na \vv_0 |^2  \ dx dy = 2 \int_a^b \int_0^{L_x} \om_0^2 \ dx dy.
\]
Define two sets:
\begin{eqnarray*}
S &=& \left \{ \vv \  \bigg | \  \int_a^b \int_0^{L_x} \om^2 \ dx dy \leq 2 \int_a^b \int_0^{L_x} \om_0^2 \ dx dy \right \} , \\
S_1 &=& \left \{ \vv \in H^s(D_c) \  \bigg | \  \int_a^b \int_0^{L_x} \om^2 \ dx dy \leq 2 \int_a^b \int_0^{L_x} \om_0^2 \ dx dy \right \} .
 \end{eqnarray*}
Notice that $S_1$ is invariant under the 2D inviscid channel flow, and $S_1$ is a dense subset of $S$,
$S_1 = S \cap H^s(D_c)$. By Lemma \ref{compe}, the closure of $S$ in $L^2(D_c)
=H^0(D_c)$ is a compact subset. For any $\vv \in S$, denote by
\[
B_{\dl /2}(\vv)=\{ \vu \in H^0(D_c) \  | \  \| \vu -\vv \|_{H^0(D_c)}
< \dl /2\} .
\]
All these balls $\{ B_{\dl /2}(\vv ) \}_{\vv \in S}$ is an open cover of the closure of $S$ in 
$H^0(D_c)$. Thus there is a finite subset $ \{ \vv_1, \cdots ,\vv_N \} \subset S$ such that 
$\{ B_{\dl /2}(\vv_n) \}_{n=1, \cdots , N}$ is a finite cover. Since $S_1$ is dense in $S$, for each such 
$\vv_n$, one can find a $\vv^{\ *}_n \in S_1$ such that 
\[
\| \vv_n - \vv^{\ *}_n  \|_{H^0(D_c)} \leq \|  \vv_n - \vv^{\ *}_n  \|_{W^{1,2}(D_c)}  < \dl /4,
\]
by the Poincar\'e inequality, where again the $W^{1,2}(D)$ norm is equivalent to the vorticity $L^2$ norm in $S$, by Lemma \ref{equ}.
All the balls
\[
B_{\dl }(\vv^{\ *}_n)=\{ \vv \in H^0(D_c) \  | \  \| \vv -\vv^{\ *}_n \|_{H^0(D_c)} < \dl \}
\]
still covers $S$, thus covers $S_1 = S \cap H^s(D_c)$. Let $B^0_{\dl }(\vv^{\ *}_n)=
B_{\dl }(\vv^{\ *}_n) \cap H^s(D_c)$,
\[
B^0_{\dl }(\vv^{\ *}_n)= \{ \vv \in H^s(D_c) \ | \   \| \vv - \vv^{\ *}_n \|_{H^0(D_c)} < \dl \} .
\]
Then 
\[
S_1 \subset \bigcup_{n=1}^N B^0_{\dl }(\vv^{\ *}_n) .
\]
By the invariance of $S_1$ under the 2D inviscid channel flow $F^t$, there is at least one $n$ such that 
an infinite subsequence of $\{ F^{mT}(\vv_0 )\}_{m=0,1, \cdots }$ is included in $B^0_{\dl }(\vv^{\ *}_n) $.
The theorem is proved.
\end{proof}

\end{document}